\documentclass[11pt,twoside, final]{amsart}
\copyrightinfo{0}{Iranian Mathematical Society}
\pagespan{1}{\pageref*{LastPage}}
\usepackage{etoolbox,lastpage}
\commby{}
\date{\scriptsize   Received: , Accepted: .}
\usepackage{amsmath,amsthm,amscd,amsfonts,amssymb,enumerate}
\usepackage{graphicx}		
\usepackage{color}
\usepackage[colorlinks]{hyperref}
 
\newtheorem{theorem}{Theorem}[section]
\newtheorem{proposition}[theorem]{Proposition}
\newtheorem{lemma}[theorem]{Lemma}
\newtheorem{corollary}[theorem]{Corollary}
\theoremstyle{definition}
\newtheorem{definition}[theorem]{Definition}
\newtheorem{example}[theorem]{Example}
\newtheorem{remark}[theorem]{Remark}
\theoremstyle{remark}

\DeclareMathOperator{\Sp}{Sp}
\numberwithin{equation}{section}
 
\begin{document}

 
\title[Hermite-Hadamard type inequality for  operator $G$-convex]{Hermite-Hadamard type inequalities for operator geometrically convex functions II } 
 
\author[A. Taghavi]{Ali Taghavi$^*$}
\address[Ali Taghavi]{Department of Mathematics, Faculty of Mathematical
Sciences, University of Mazandaran, P. O. Box 47416-1468,
Babolsar, Iran.}
\email{taghavi@umz.ac.ir}

\author[V. Darvish]{Vahid Darvish }
\address[Vahid Darvish]{Department of Mathematics, Faculty of Mathematical
Sciences, University of Mazandaran, P. O. Box 47416-1468,
Babolsar, Iran.}
\email{vahid.darvish@mail.com}

\author[T. A. Roushan]{Tahere Azimi Roushan }
\address[Tahere Azimi Roushan]{Department of Mathematics, Faculty of Mathematical
Sciences, University of Mazandaran, P. O. Box 47416-1468,
Babolsar, Iran.}
\email{t.roushan@umz.ac.ir}

  \thanks{$^*$Corresponding author}
%
 
 \maketitle
%

\begin{abstract}
In this paper, we prove some Hermite-Hadamard type inequalities for operator geometrically convex functions for non-commutative operators.
\end{abstract}
 
\section{Introduction and preliminaries}
Let $B(H)$ stand for $C^{*}$-algebra of  all bounded linear
operators on a complex Hilbert space $H$ with inner
product $\langle \cdot,\cdot\rangle$. An operator $A\in B(H)$ is strictly positive and write $A>0$ if $\langle
Ax,x\rangle>0$ for all $x\in H$. Let $B(H)^{++}$ stand for all strictly positive operators on $B(H)$.

Let $A$ be a self-adjoint operator in $B(H)$. The Gelfand map establishes a $\ast$-isometrically isomorphism $\Phi$ between the set $C(\Sp(A))$ of all continuous functions defined on the spectrum of $A$, denoted $\Sp(A)$, and the $C^{*}$-algebra $C^{*}(A)$ generated by $A$ and the identity operator $1_{H}$ on $H$ as follows:

For any $f,g\in C(\Sp(A)))$ and any $\alpha, \beta\in\mathbb{C}$ we have:
\begin{itemize}
\item
$\Phi(\alpha f+\beta g)=\alpha \Phi(f)+\beta \Phi(g);$
\item
$\Phi(fg)=\Phi(f)\Phi(g)$ and $ \Phi(\bar{f})=\Phi(f)^{*};$
\item
$\|\Phi(f)\|=\|f\|:=\sup_{t\in \Sp(A)}|f(t)|;$
\item
$\Phi(f_{0})=1_{H}$ and $\Phi(f_{1})=A,$ where $f_{0}(t)=1$ and $f_{1}(t)=t$, for $t\in \Sp(A)$.
\end{itemize}
with this notation we define
$$f(A)=\Phi(f) \ \text{for all} \ \ f\in C(\Sp(A))$$
 and we call it the continuous functional calculus for a self-adjoint operator $A$.
 
 If $A$ is a self-adjoint operator and  both $f$ and $g$ are real valued functions on $\Sp(A)$ then the following important property holds:
\begin{equation}\label{e3}
f(t)\geq g(t) \ \ \text{for any} \ \ t\in \Sp(A)\ \ \text{implies that} \ \ f(A)\geq g(A),
\end{equation}
in the operator order of $B(H)$, see \cite{zhu}.

A real valued continuous function $f:\mathbb{R}\to\mathbb{R}$ is said to be convex (concave) if 
$$f(\lambda a+(1-\lambda)b)\leq (\geq) \lambda f(a)+(1-\lambda) f(b)$$
for $a,b\in \mathbb{R}$ and $\lambda\in [0,1]$.

The following Hermite-Hadamard inequality holds for any convex function $f$ defined on $\mathbb{R}$
\begin{eqnarray}
(b-a)f\left(\frac{a+b}{2}\right)&\leq& \int_{a}^{b} f(x)dx\nonumber\\
&\leq& (b-a)\frac{f(a)+f(b)}{2}, \ \ \text{for} \ a,b\in\mathbb{R}.\label{ro}
\end{eqnarray}
The author of \cite[Remark 1.9.3]{nic2} gave the following refinement of Hermite-Hadamard inequalities for convex functions
\begin{eqnarray*}
f\left(\frac{ a+b}{2}\right)&\leq&\frac{1}{2}\left(f\left(\frac{3 a+ b}{4}\right)+f\left(\frac{ a+3 b}{4}\right)\right)\\
&\leq&\frac{1}{ b- a}\int_{ a}^{ b} f(x)dx\\
&\leq&\frac{1}{2}\left(f\left(\frac{ a+ b}{2}\right)+\frac{f( a)+f( b)}{2}\right)\\
&\leq&\frac{f( a)+f( b)}{2}.
\end{eqnarray*}

A real valued continuous function is operator convex if $$f(\lambda A+(1-\lambda)B)\leq \lambda f(A)+(1-\lambda)f(B)$$ for self-adjoint operator $A, B\in B(H)$ and $\lambda\in[0,1]$.

 In \cite{dra3} Dragomir investigated the operator version of the Hermite-Hadamard inequality for
operator convex functions. Let $f:\mathbb{R}\to\mathbb{R}$ be an operator convex function on the interval $I$ then, for any self-adjoint operators $A$ and $B$ with spectra in $I$, the following inequalities holds
\begin{eqnarray}
f\left(\frac{A+B}{2}\right)&\leq&2\int_{\frac{1}{4}}^{\frac{3}{4}} f(tA+(1-t)B)dt\\
&\leq& \frac{1}{2}\left[f\left(\frac{3A+B}{4}\right)+f\left(\frac{A+3B}{4}\right)\right]\\
&\leq& \int_{0}^{1}f\left((1-t)A+tB\right)dt\nonumber\\
&\leq&\frac{1}{2}\left[f\left(\frac{A+B}{2}\right)+\frac{f(A)+f(B)}{2}\right]\\
&\leq& \frac{f(A)+f(B)}{2}, \label{kdv}
\end{eqnarray} 
for the first inequality in above, see \cite{taghavi}.

A continuous function $f: I\subseteq \mathbb{R}^{+}\to\mathbb{R}^{+}$, ($\mathbb{R}^{+}$ denoted positive real numbers) is said to be geometrically convex function (or multiplicatively  convex function) if 
$$f(a^{\lambda}b^{1-\lambda})\leq f(a)^{\lambda}f(b)^{1-\lambda}$$
for $a,b\in I$ and $\lambda\in[0,1]$.

The author of  \cite[p. 158]{nic} showed that every polynomial $P(x)$ with non-negative coefficients is a geometrically convex function on $[0,\infty)$. More generally, every real analytic function $f(x)=\sum_{n=0}^{\infty}c_{n}x^{n}$ with non-negative coefficients is geometrically convex function on $(0,R)$ where $R$ denotes the radius of convergence. Also, see \cite{taghavi2,tag}.

In \cite{taghavi}, the following inequalities were obtained for a geometrically convex function
\begin{eqnarray*}
f(\sqrt{ab})&\leq& \sqrt{\left(f(a^{\frac{3}{4}}b^{\frac{1}{4}})f(a^{\frac{1}{4}}b^{\frac{3}{4}})\right)}\\
&\leq& \exp\left(\frac{1}{\log b-\log a}\int_{a}^{b}\frac{\log f(t)}{t}dt\right)\\
&\leq&\sqrt{f(\sqrt{ab})}.\sqrt[4]{f(a)}.\sqrt[4]{f(b)}\\
&\leq& \sqrt{f(a)f(b)}
\end{eqnarray*}

In this paper, we prove some  Hermite-Hadamard inequalities for operator geometrically convex functions. Moreover, in the final section, we present some examples and remarks.

\section{Hermite-Hadamard inequalities for geometrically convex functions}
In this section, we introduce the concept of operator geometrically convex function for positive operators and prove the Hermite-Hadamard type inequalities for this function.

\begin{proposition}\label{tc}
Let $A, B\in B(H)^{++}$ such that $Sp(A), Sp(B)\subseteq I$, then $Sp(A\sharp_{t} B)\subseteq I$, where $A\sharp_{t} B=A^{\frac{1}{2}}{(A^{-\frac{1}{2}}BA^{-\frac{1}{2}})}^{t}A^{\frac{1}{2}}$ is $t$-geometric mean.
\end{proposition}
\begin{proof}
Let $I=[m, M]$ for some positive real numbers $m, M$ with $m<M$. Since $Sp(A),Sp(B) \subseteq I$ it is equivalent to $m\leq A\leq M$ and $m\leq B\leq M$. So, by virtue of operator monotonicity property of the function $f(x)=x^{t}$ on $(0,\infty)$ for $t\in[0,1]$, and by using the fact that if $a, b$ be self-adjoint operators in $C^{*}$-algebra $\mathcal{A}$, if $a\leq b$ and $c\in \mathcal{A}$ then $c^{*}ac\leq c^{*}bc$, we get the result.
\end{proof}
Now, by applying Proposition \ref{tc}, we present the following definition.
\begin{definition}\label{refe}
A continuous function $f:I\subseteq\mathbb{R}^{+}\to\mathbb{R}^{+}$ is said to be  operator geometrically convex  if $$f(A\sharp_{t}B)\leq f(A)\sharp_{t}f(B)$$ for $A, B\in B(H)^{++}$ such that $\Sp(A), \Sp(B)\subseteq I$.
\end{definition}
We need the following lemmas for proving our theorems.
\begin{lemma}\cite{kub, law}\label{hagh}
Let $A, B\in B(H)^{++}$ and let $t,s,u\in\mathbb{R}$. Then
\begin{equation}
(A\sharp_{t} B)\sharp_{s}(A\sharp_{u} B)=A\sharp_{(1-s)t+su} B.
\end{equation}
\end{lemma}

\begin{lemma}\cite{kub}\label{hagh2}
Let $A$, $B$, $C$ and $D$ be operators in $B(H)^{++}$   and let $t\in\mathbb{R}$. Then, we have
\begin{equation}
A\sharp_{t}B\leq C\sharp_{t} D
\end{equation}
for $A\leq C$ and $B\leq D$.
\end{lemma}

\begin{lemma}\label{bim}
Let $A, B\in B(H)^{++}$. If $f:I\subseteq\mathbb{R}^{+}\to\mathbb{R}^{+}$ is a continuous function, then
\begin{equation*}\label{rav}
\int_{0}^{1}f\left(A \sharp_{t} B \right)\sharp f\left(A \sharp_{1-t} B \right)dt \leq\left( \int_{0}^{1}f\left(A \sharp_{t} B \right)dt\right) \sharp\left(\int_{0}^{1} f\left(A \sharp_{1-t} B \right)dt \right)
\end{equation*}
such that $\Sp(A), \Sp(B)\subseteq I$.
\end{lemma}
\begin{proof}
Since the function $t^{\frac{1}{2}}$ is operator concave, we can write
\begin{small}
\begin{align*}
&\left(\left(\int_{0}^{1}f(A \sharp_{1-u} B)du\right)^{\frac{-1}{2}}\left(\int_{0}^{1}f(A \sharp_{u} B) du\right)\left(\int_{0}^{1}f(A \sharp_{1-u} B) du\right)^{\frac{-1}{2}}\right)^{\frac{1}{2}}\\
& \  \ \ \ \ \ \ \ \ \ \ \ \ \ \ \ \ \  \text {By change of variable $v=1-u$}\\
&=\left(\left(\int_{0}^{1}f(A \sharp_{v} B)dv\right)^{\frac{-1}{2}}\left(\int_{0}^{1}f(A \sharp_{u} B) du\right)\left(\int_{0}^{1}f(A \sharp_{v} B) dv\right)^{\frac{-1}{2}}\right)^{\frac{1}{2}}\\
&=\left(\int_{0}^{1}\left(\int_{0}^{1}f(A \sharp_{v} B)dv\right)^{\frac{1}{2}}f(A \sharp_{u} B)\left(\int_{0}^{1}f(A \sharp_{v} B) dv\right)^{\frac{1}{2}} du \right)^{\frac{1}{2}}\\
\begin{split}
&=\left(\int_{0}^{1}\left(\int_{0}^{1}f(A \sharp_{v} B)dv\right)^{\frac{-1}{2}}(f(A \sharp_{1-u} B))^{\frac{1}{2}}
\left((f(A \sharp_{1-u} B))^{\frac{-1}{2}}f(A \sharp_{u} B)(f(A \sharp_{1-u} B))^{\frac{-1}{2}}\right)\right. \\
&\quad \left. {}\times(f(A \sharp_{1-u} B))^{\frac{1}{2}} \left(\int_{0}^{1}f(A \sharp_{v} B) dv\right)^{\frac{-1}{2}} du\right)^{\frac{1}{2}}
\end{split}\\
& \ \ \ \ \ \   \ \ \ \ \  \ \ \ \ \  \text {By the operator Jensen inequality}
\end{align*}
\end{small}
\begin{small}
\begin{align*}
&\geq\int_{0}^{1}\left(\int_{0}^{1}f(A \sharp_{v} B)dv\right)^{\frac{-1}{2}}(f(A \sharp_{1-u} B))^{\frac{1}{2}}
\left(\left(f(A \sharp_{1-u} B))^{\frac{-1}{2}}f(A \sharp_{u} B)(f(A \sharp_{1-u} B)\right)^{\frac{-1}{2}}\right)^{\frac{1}{2}} \\
&\times(f(A \sharp_{1-u} B))^{\frac{1}{2}} \left(\int_{0}^{1}f(A \sharp_{v} B) dv\right)^{\frac{-1}{2}} du\\
&=\left(\int_{0}^{1}f(A \sharp_{v} B)dv\right)^{\frac{-1}{2}}\int_{0}^{1}(f(A \sharp_{1-u} B))^{\frac{1}{2}}
\left(\left(f(A \sharp_{1-u} B))^{\frac{-1}{2}}f(A \sharp_{u} B)(f(A \sharp_{1-u} B)\right)^{\frac{-1}{2}}\right)^{\frac{1}{2}} \\
&\times(f(A \sharp_{1-u} B))^{\frac{1}{2}}  du\left(\int_{0}^{1}f(A \sharp_{v} B) dv\right)^{\frac{-1}{2}} \\
& \  \ \ \ \ \ \ \ \ \ \ \ \ \ \ \ \ \  \text {By change of variable $u=1-v$ }\\
&=\left(\int_{0}^{1}f(A \sharp_{1-u} B)du\right)^{\frac{-1}{2}}\int_{0}^{1}(f(A \sharp_{1-u} B))^{\frac{1}{2}}
\left(\left(f(A \sharp_{1-u} B))^{\frac{-1}{2}}f(A \sharp_{u} B)(f(A \sharp_{1-u} B)\right)^{\frac{-1}{2}}\right)^{\frac{1}{2}} \\
&\times(f(A \sharp_{1-u} B))^{\frac{1}{2}}  du\left(\int_{0}^{1}f(A \sharp_{1-u} B) du\right)^{\frac{-1}{2}}. 
\end{align*}
\end{small}
So, we obtain
\begin{small}
\begin{align*}
&\left(\left(\int_{0}^{1}f(A \sharp_{1-u} B)du\right)^{\frac{-1}{2}}\left(\int_{0}^{1}f(A \sharp_{u} B) du\right)\left(\int_{0}^{1}f(A \sharp_{1-u} B) du\right)^{\frac{-1}{2}}\right)^{\frac{1}{2}}\\
&\geq \left(\int_{0}^{1}f(A \sharp_{1-u} B)du\right)^{\frac{-1}{2}}\int_{0}^{1}(f(A \sharp_{1-u} B))^{\frac{-1}{2}}
\left(\left(f(A \sharp_{1-u} B))^{\frac{-1}{2}}f(A \sharp_{u} B)(f(A \sharp_{1-u} B)\right)^{\frac{-1}{2}}\right)^{\frac{1}{2}} \\
&\times(f(A \sharp_{1-u} B))^{\frac{1}{2}}  du\left(\int_{0}^{1}f(A \sharp_{1-u} B) du\right)^{\frac{-1}{2}}.
\end{align*} 
\end{small}
Multiplying both side of the above inequality by $\left(\int_{0}^{1}f(A \sharp_{1-u} B) du\right)^{\frac{1}{2}}$ to obtain
\begin{equation*}
\left( \int_{0}^{1}f\left(A \sharp_{u} B \right)du\right) \sharp\left(\int_{0}^{1} f\left(A \sharp_{1-u} B \right)du \right)\geq \int_{0}^{1}f\left(A \sharp_{u} B \right)\sharp f\left(A \sharp_{1-u} B \right)du.
\end{equation*}
\end{proof}
Before giving our theorems in this section, we mention the following remark.
\begin{remark}\label{rem12}
Let  $p\left(x\right)=x^{t}$ and $q\left(x\right)=x^{s}$ on $[1,\infty )$, where $0\leq t\leq s$. If $f\left(A\right)\leq f\left(B\right)$ then $\Sp \left(f\left(A\right)^{\frac{-1}{2}}\left(f(B)\right)f\left(A\right)^{\frac{-1}{2}}\right)\subseteq [1,\infty)$. By functional calculus, we have
$$p\left(f\left(A\right)^{\frac{-1}{2}}f\left(B\right)f\left(A\right)^{\frac{-1}{2}}\right)\leq q\left(f\left(A\right)^{\frac{-1}{2}}f\left(B\right)f\left(A\right)^{\frac{-1}{2}}\right).$$
So, $$\left(f\left(A\right)^{\frac{-1}{2}}f\left(B\right)f\left(A\right)^{\frac{-1}{2}}\right)^{t}\leq\left(f\left(A\right)^{\frac{-1}{2}}f\left(B\right)f\left(A\right)^{\frac{-1}{2}}\right)^{s}.$$
\end{remark}

Now, we are ready to prove Hermite-Hadamard type inequality for operator geometrically convex functions.
\begin{theorem}\label{thm}
Let $f$ be an operator geometrically convex function. Then, we have
\begin{equation}\label{mr}
f\left(A \sharp B \right)\leq \int_{0}^{1} f\left(A \sharp_{t} B \right)dt\leq \int_{0}^{1} f(A)\sharp_{t} f(B) dt.
\end{equation}
Moreover, if $f(A)\leq f(B)$, then we have
\begin{equation}\label{mr123}
\int_{0}^{1}f(A\sharp_{t}B) dt\leq \int_{0}^{1} f(A)\sharp_{t} f(B) dt\leq \frac{1}{2}((f(A)\sharp f(B))+f(B))
\end{equation}
for $A ,B \in B(H)^{++}$. 
\end{theorem}
\begin{proof}
Let $f$ be a geometrically convex function then we have
\begin{eqnarray*}
f\left(A \sharp B \right)&=&f\left(\left(A \sharp_{t} B \right)\sharp \left(A \sharp_{1-t}B \right)\right) \ \ \ \ \ \text{By Lemma \ref{hagh}}\\
&\leq& f\left(A  \sharp_{t} B \right)\sharp f\left(A \sharp_{1-t} B \right) \ \ \ \ \ \text{$f$ is operator geometrically convex.}
\end{eqnarray*}
Taking integral of the both sides of the above inequalities on $[0,1]$ we obtain
\begin{eqnarray*}
f\left(A \sharp B \right)&\leq&\int_{0}^{1}f\left(A \sharp_{t} B \right)\sharp f\left(A \sharp_{1-t} B \right)dt \\
&\leq&\left( \int_{0}^{1}f\left(A \sharp_{t} B \right)dt\right) \sharp\left(\int_{0}^{1} f\left(A \sharp_{1-t} B \right)dt \right)\ \ \ \text{By Lemma \ref{bim}}\\
&=& \int_{0}^{1} f\left(A \sharp_{t} B \right)dt \\
&\leq& \int_{0}^{1} f\left(A \right)\sharp_{t} f\left(B \right)dt. 
\end{eqnarray*}
For the case $f(A)\leq f(B)$, by applying Remark \ref{rem12} for $s=\frac{1}{2}$ we have
$$\left({f(A)}^{-\frac{1}{2}}f(B){f(A)}^{-\frac{1}{2}}\right)^{t}\leq \left({f(A)}^{-\frac{1}{2}}f(B){f(A)}^{-\frac{1}{2}}\right)^{\frac{1}{2}}.$$
By integrating the above inequality over $t\in [0,\frac{1}{2}]$, we obtain
$$\int_{0}^{\frac{1}{2}}\left({f(A)}^{-\frac{1}{2}}f(B){f(A)}^{-\frac{1}{2}}\right)^{t}dt\leq \frac{1}{2}\left({f(A)}^{-\frac{1}{2}}f(B){f(A)}^{-\frac{1}{2}}\right)^{\frac{1}{2}}.$$
Multiplying both sides of the above inequality by $f(A)^{\frac{1}{2}}$, we have
\begin{eqnarray*}
&&\int_{0}^{\frac{1}{2}}f(A)^{\frac{1}{2}}\left({f(A)}^{-\frac{1}{2}}f(B){f(A)}^{-\frac{1}{2}}\right)^{t}f(A)^{\frac{1}{2}}dt\\&&\leq \frac{1}{2}\left(f(A)^{\frac{1}{2}}\left({f(A)}^{-\frac{1}{2}}f(B){f(A)}^{-\frac{1}{2}}\right)^{\frac{1}{2}}f(A)^{\frac{1}{2}}\right).
\end{eqnarray*}

It follows that
\begin{equation}\label{ta}
\int_{0}^{\frac{1}{2}}f(A)\sharp_{t}f(B)\leq\frac{f(A)\sharp f(B)}{2}.
\end{equation}
On the other hand, by considering Remark \ref{rem12} for $s=1$ we have
$$\left({f(A)}^{-\frac{1}{2}}f(B){f(A)}^{-\frac{1}{2}}\right)^{t}\leq {f(A)}^{-\frac{1}{2}}f(B){f(A)}^{-\frac{1}{2}}.$$
Integrating the above inequality over $t\in [\frac{1}{2},1]$, we get
$$\int_{\frac{1}{2}}^{1}\left({f(A)}^{-\frac{1}{2}}f(B){f(A)}^{-\frac{1}{2}}\right)^{t}dt\leq \frac{1}{2}\left({f(A)}^{-\frac{1}{2}}f(B){f(A)}^{-\frac{1}{2}}\right).$$
By multiplying both side of the above inequality by $f(A)^{\frac{1}{2}}$, we have
$$\int_{\frac{1}{2}}^{1}f(A)^{\frac{1}{2}}\left({f(A)}^{-\frac{1}{2}}f(B){f(A)}^{-\frac{1}{2}}\right)^{t}f(A)^{\frac{1}{2}}dt\leq \frac{f(B)}{2}.$$
It follows that
\begin{equation}\label{tb}
\int_{\frac{1}{2}}^{1}f(A)\sharp_{t}f(B)\leq\frac{f(B)}{2}.
\end{equation}
From inequalities (\ref{ta}) and (\ref{tb}) we obtain 
\begin{eqnarray*}
\int_{0}^{\frac{1}{2}}f(A\sharp_{t}B)dt+\int_{\frac{1}{2}}^{1}f(A\sharp_{t}B)dt&\leq&\int_{0}^{\frac{1}{2}}f(A)\sharp_{t}f(B)dt+\int_{\frac{1}{2}}^{1}f(A)\sharp_{t}f(B)dt\\
&\leq&\frac{f(A)\sharp f(B)}{2}+\frac{f(B)}{2}.
\end{eqnarray*}
It follows that
$$\int_{0}^{1}f(A\sharp_{t}B)dt\leq\int_{0}^{1}f(A)\sharp_{t}f(B)dt\leq\frac{1}{2}((f(A)\sharp f(B))+f(B)).$$
\end{proof}

By making use of inequalities (\ref{mr}) and (\ref{mr123}), we have the following result.
\begin{corollary}\label{narenj}
Let $f$ be an operator geometrically convex function. Then, if $f(A)\leq f(B)$  we have
\begin{equation}\label{mr222}
f(A\sharp B)\leq \int_{0}^{1} f\left(A\sharp_{t} B\right) dt  \leq \frac{1}{2}\left((f(A)\sharp f(B))+f(B)\right).
\end{equation}
for $A,B\in B(H)^{++}$.
\end{corollary}

\begin{theorem}\label{mche}
Let $f$ be an operator geometrically convex function. Then, we have
\begin{equation}
f\left(A \sharp B \right)\leq \int_{0}^{1} f\left(A \sharp_{t} B \right)\sharp f\left(A\sharp_{1-t} B\right) dt\leq f(A)\sharp f(B)
\end{equation}
for $A ,B \in B(H)^{++}$.
\end{theorem}
\begin{proof}
We can write
\begin{eqnarray*}
f(A\sharp B)&=& f\left((A\sharp_{t} B)\sharp (A\sharp_{1-t}B)\right) \ \ \ \text{By Lemma \ref{hagh}}\\
&\leq& f(A\sharp_{t} B)\sharp f(A\sharp_{1-t} B)\  \ \ \ \text{$f$ is operator geometrically convex}\\
&\leq& \left(f(A)\sharp_{t} f(B)\right)\sharp\left(f(A)\sharp_{1-t} f(B)\right)  \ \ \ \text{By Lemma \ref{hagh2}}\\
&=& f(A)\sharp f(B).
\end{eqnarray*}
So, we obtain
$$f(A\sharp B)\leq f(A\sharp_{t} B)\sharp f(A\sharp_{1-t} B)\leq f(A)\sharp f(B).$$
Integrating the above inequality over $t\in[0,1]$ we obtain
the desired result.
\end{proof}

We divide the interval $[0,1]$ to the interval $[\nu,1-\nu]$ when $\nu\in[0,\frac{1}{2})$ and to the  interval $[1-\nu,\nu]$ when $\nu\in (\frac{1}{2},1]$. The we have the following inequalities.
\begin{theorem}\label{sta}
Let $A,B\in B(H)^{++}$ such that $f(A)\leq f(B)$. Then, we have
\begin{enumerate}
\item For $\nu\in[0,\frac{1}{2})$
\begin{equation}\label{sys1}
f(A)\sharp_{\nu}f(B)\leq \frac{1}{1-2\nu}\int_{\nu}^{1-\nu}f(A)\sharp_{t} f(B) dt\leq f(A)\sharp_{1-\nu} f(B).
\end{equation}

\item
For $\nu\in(\frac{1}{2},1]$
\begin{equation}\label{sys2}
f(A)\sharp_{1-\nu}f(B)\leq \frac{1}{2\nu-1}\int_{1-\nu}^{\nu}f(A)\sharp_{t} f(B) dt\leq f(A)\sharp_{\nu} f(B).
\end{equation}
\end{enumerate}
\end{theorem}
\begin{proof}
Let $\nu\in[0,\frac{1}{2})$, then by Remark \ref{rem12} we have
\begin{eqnarray*}
\left(f(A)^{\frac{-1}{2}}f(B)f(A)^{\frac{-1}{2}}\right)^{\nu}&\leq&\left(f(A)^{\frac{-1}{2}}f(B)f(A)^{\frac{-1}{2}}\right)^{t}\\
&\leq& \left(f(A)^{\frac{-1}{2}}f(B)f(A)^{\frac{-1}{2}}\right)^{1-\nu}
\end{eqnarray*}
for $\nu\leq t\leq 1-\nu$ and $A ,B \in B(H)^{++}$ such that $\Sp \left(A \right), \Sp\left(B \right)\subseteq I$.\\
By integrating the above inequality over $t\in[\nu,1-\nu]$ we obtain
\begin{eqnarray*}
\int_{\nu}^{1-\nu}\left(f(A)^{\frac{-1}{2}}f(B)f(A)^{\frac{-1}{2}}\right)^{\nu} dt&\leq&\int_{\nu}^{1-\nu}\left(f(A)^{\frac{-1}{2}}f(B)f(A)^{\frac{-1}{2}}\right)^{t} dt\\
& \leq&\int_{\nu}^{1-\nu}\left(f(A)^{\frac{-1}{2}}f(B)f(A)^{\frac{-1}{2}}\right)^{1-\nu} dt.
\end{eqnarray*}
It follows that
\begin{eqnarray*}
\left(f(A)^{\frac{-1}{2}}f(B)f(A)^{\frac{-1}{2}}\right)^{\nu}&\leq&
\frac{1}{1-2\nu}\int_{\nu}^{1-\nu}\left(f(A)^{\frac{-1}{2}}f(B)f(A)^{\frac{-1}{2}}\right)^{t} dt\\
&\leq& \left(f(A)^{\frac{-1}{2}}f(B)f(A)^{\frac{-1}{2}}\right)^{1-\nu}.
\end{eqnarray*}
Multiplying the both sides of the above inequality by $f(A)^{\frac{1}{2}}$ gives us
$$f(A)\sharp_{\nu}f(B)\leq \frac{1}{1-2\nu}\int_{\nu}^{1-\nu}f(A)\sharp_{t} f(B) dt\leq f(A)\sharp_{1-\nu} f(B).$$
Also, we know that
\begin{eqnarray*}
\lim_{\nu\to\frac{1}{2}} f(A)\sharp_{\nu}f(B)&=&\lim_{\nu\to\frac{1}{2}} \frac{1}{1-2\nu}\int_{\nu}^{1-\nu}f(A)\sharp_{t} f(B) dt\\
&=&\lim_{\nu\to\frac{1}{2}}f(A)\sharp_{1-\nu} f(B)\\
& =&f(A)\sharp f(B).
\end{eqnarray*}
Similarily, for $\nu\in(\frac{1}{2},1]$, by a same proof as above we get
$$f(A)\sharp_{1-\nu} f(B)\leq\frac{1}{2\nu-1}\int_{1-\nu}^{\nu}f(A)\sharp_{t} f(B) dt\leq  f(A)\sharp_{\nu} f(B).$$
\end{proof}
By definition of geometrically convex function and (\ref{sys1}) we have
\begin{eqnarray*}
f\left(A \sharp_{\nu} B \right)&\leq& \frac{1}{1-2\nu} \int_{\nu}^{1-\nu} f\left(A \sharp_{t} B \right)dt\\
&\leq& \frac{1}{1-2\nu} \int_{\nu}^{1-\nu} f(A)\sharp_{t} f(B) dt\\
&\leq&f(A)\sharp_{1-\nu} f(B)
\end{eqnarray*}
for $\nu\in[0,\frac{1}{2})$. We should mention here that
$$\lim_{\nu\to \frac{1}{2}}\frac{1}{1-2\nu} \int_{\nu}^{1-\nu} f\left(A \sharp_{t} B \right)dt=\lim_{\nu\to\frac{1}{2}}f(A\sharp_{\nu}B)=f(A\sharp B).$$
On the other hand, by definition of geometrically convex function and (\ref{sys2}) we have
\begin{eqnarray*}
f\left(A \sharp_{1-\nu} B \right)&\leq& \frac{1}{2\nu-1} \int_{1-\nu}^{\nu} f\left(A \sharp_{t} B \right)dt\\
&\leq& \frac{1}{2\nu-1} \int_{1-\nu}^{\nu} f(A)\sharp_{t} f(B) dt\\
&\leq&f(A)\sharp_{\nu} f(B)
\end{eqnarray*}
for $\nu\in(\frac{1}{2},1]$.

\section{Examples and remarks}
In this section we give some examples of the results that obtained in the previous section.
\begin{remark}
For positive $A,B\in B(H)$, Ando proved in \cite{and} that if $\Psi$ is a positive linear map, then we have
$$\Psi(A\sharp B)\leq \Psi(A)\sharp \Psi(B).$$
The above inequality shows that we can find some examples for Definition \ref{refe} when $f$ is linear.
\end{remark}
\begin{example}
It is easy to check that the function $f(t)=t^{-1}$ is operator geometrically convex for operators in $B(H)^{++}$.
\end{example}
 
\begin{definition} Let $\phi$ be a map on $C^{*}$-algebra $B(H)$. We say that $\phi$ is $2$-positive if the $2\times2$ operator matrix 
$\begin{bmatrix}
A & B \\ B^{*} & C 
\end{bmatrix}\geq0$
then we have
$\begin{bmatrix}
\phi(A) & \phi(B) \\  \phi(B^{*}) & \phi(C)
\end{bmatrix}\geq0$.
\end{definition}
In \cite{lin}, M. Lin gave an example of a $2$-positive map over contraction operators (i.e., $\|A\|<1$). He proved that 
\begin{equation}\label{referee}
\phi(t)=(1-t)^{-1}
\end{equation}
is $2$-positive.

\begin{example}\label{disg}
Let $A$ and $B$ be two contraction operators in $B(H)^{++}$. Then it is easy to check $A\sharp B$ is also a contraction and positive. Also, we know the $2\times 2$ operator matrix 
$$\begin{bmatrix}
A & A\sharp B \\ A\sharp B & B 
\end{bmatrix}$$
is semidefinite positive.
Hence, by (\ref{referee}) we obtain 
$$\begin{bmatrix}
(I-A)^{-1} & (I-(A\sharp B))^{-1} \\ (I-(A\sharp B))^{-1} & (I-B)^{-1}
\end{bmatrix}$$
is semidefinite positive. 

On the other hand, by Ando's characterization of the geometric mean if $X$ is a Hermitian matrix and 
$$\begin{bmatrix}
A & X \\ X & B 
\end{bmatrix}\geq0 , $$
then $X\leq A\sharp B$.\\ So we conclude that $(I-(A\sharp B))^{-1}\leq (I-A)^{-1}\sharp(I-B)^{-1}$. Therefore, the function $\phi(t)=(1-t)^{-1}$ is operator geometrically convex.
\end{example}
Also, Lin proved that the function $$\phi(t)=\frac{1+t}{1-t}$$ is $2$-positive over contractions. By the same argument as Example \ref{disg} we can say the above function is operator geometrically convex too.

\begin{example}
In the proof of \cite[Theorem 4.12]{jun}, by applying H\"{o}lder-McCarthy inequality the authors showed the following  inequalities
\begin{eqnarray*}
\langle A\sharp_{\alpha} B x,x\rangle &=&\left\langle \left(A^{\frac{-1}{2}}BA^{\frac{-1}{2}}\right)^{\alpha}A^{\frac{1}{2}}x,A^{\frac{1}{2}}x\right\rangle\\
&\leq& \left\langle \left(A^{\frac{-1}{2}}BA^{\frac{-1}{2}}\right)A^{\frac{1}{2}}x,A^{\frac{1}{2}}x\right\rangle^{\alpha}\left\langle A^{\frac{1}{2}}x,A^{\frac{1}{2}}x\right\rangle^{1-\alpha}\\
&=&\langle Ax,x\rangle^{1-\alpha}\langle Bx,x\rangle^{\alpha}\\
&=&\langle Ax,x\rangle \sharp_{\alpha} \langle Bx,x\rangle
\end{eqnarray*}
for $x\in H$ and $\alpha\in[0,1]$. By taking the supremum over unit vector $x$, we obtain that $f(x)=\|x\|$ is geometrically convex function for usual operator norms.
\end{example}
By the above example and Corollary \ref{narenj}, when $\|A\|\leq\|B\|$ we have
\begin{equation}
\|A\sharp B\|\leq \int_{0}^{1} \|A\sharp_{t} B\| dt  \leq \frac{1}{2}(\sqrt{\|A\|\|B\|}+\|B\|)
\end{equation}
for $A,B\in B(H)^{++}$



\end{document}